\documentclass[12pt]{amsart}
\usepackage[matrix,arrow]{xy}
\usepackage{amssymb}
\usepackage{amsmath}
\usepackage{amsfonts}
\usepackage{epsfig}
\usepackage{color}
\usepackage{epstopdf}
\usepackage{graphicx}
\usepackage{amsthm}
\usepackage{enumerate}
\usepackage[mathscr]{eucal}
\usepackage{verbatim}

\usepackage[bookmarks=true,hyperindex,pdftex,colorlinks,citecolor=blue,
linkcolor=blue,urlcolor=blue]{hyperref}
\usepackage{tikz}
\usetikzlibrary{matrix,arrows.meta}

\theoremstyle{plain}
\newtheorem{theorem}{Theorem}[section]

\newtheorem{prop}[theorem]{Proposition}
\newtheorem{lemma}[theorem]{Lemma}

\theoremstyle{definition}

\newtheorem{question}[theorem]{Question}

\newtheorem{definition}[theorem]{Definition}

\newcommand{\R}{\mathbb{R}}
\newcommand{\N}{\mathbb{N}}
\newcommand{\C}{\mathbb{C}}

\newcommand{\Lin}{\mathcal{L}}
\newcommand{\K}{\mathcal{K}}

\newcommand{\eps}{\varepsilon}
\newcommand{\lam}{\lambda}

\DeclareMathOperator{\dist}{dist}

\DeclareMathOperator{\re}{Re}

\DeclareMathOperator{\co}{co}

\DeclareMathOperator{\NA}{NA}
\DeclareMathOperator{\MA}{MA}

\DeclareMathOperator{\NRA}{NRA}
\DeclareMathOperator{\CNA}{CNA}

\title[On the Crawford number attaining operators]
{On the Crawford number attaining operators}

\author[G.~Choi]{Geunsu Choi}
\address[G.~Choi]{Department of Mathematics Education, Dongguk University, Seoul 04620, Republic of Korea}
\email{\texttt{chlrmstn90@gmail.com}}

\author[H.~J.~Lee]{Han Ju Lee}
\address[H.~J.~Lee]{Department of Mathematics Education, Dongguk University, Seoul 04620, Republic of Korea}
\email{\texttt{hanjulee@dgu.ac.kr}}

\thanks{The first and second authors were supported by Basic Science Research Program through the National Research Foundation of
Korea(NRF) funded by the Ministry of Education, Science and Technology [NRF-2020R1A2C1A01010377].}


\date{\today}                                           

\begin{document}

\begin{abstract}
We study the denseness of Crawford number attaining operators on Banach spaces. Mainly, we prove that if a Banach space has the RNP, then the set of Crawford number attaining operators is dense in the space of bounded linear operators. We also see among others that the set of Crawford number attaining operators may be dense in the space of all bounded linear operators while they do not coincide, by observing the case of compact operators when the Banach space has a 1-unconditional basis. Furthermore, we show a Bishop-Phelps-Bollob\'as type property for the Crawford number for certain Banach spaces, and we finally discuss some difficulties and possible problems on the topic.
\end{abstract}

\maketitle

\section{Introduction}
The study of numerical radius was one of main interest in the field of functional analysis in recent decades. In particular, phenomenon of numerical radius attaining operators was discovered in many sources such as \cite{A1,A2,AAP,AP,CMM,P}. It has not been that long since the minimum norm of an operator became an issue as a separated topic. An analogue of numerical radius in terms of minimum norm, which is known as the \emph{Crawford number}, plays an important role in view of eigenvalue optimizations of matrices. This concept was first introduced in \cite{C3} and the terminology originated from \cite{KLV}. We refer to \cite{BR,C1,C2,SPM} for recent works on the field of minimum norm attaining operators, and see also \cite{SMBP} which concerns the Crawford number attainment on a Hilbert space. In \cite{C2}, the main perspective is to find conditions on pairs of Banach spaces such that the set of minimum norm attaining operators is dense in the space of all bounded linear operators. Our main goal is to study the minimum analogue of numerical radius in this sense.

As our main conventions require some universal notations, we illustrate them as follows. Let $X$ and $Y$ always denote a Banach space on a scalar field $\mathbb{K}=\R$ or $\C$ unless specifically mentioned. The unit sphere and unit ball of $X$ are denoted by $S_X$ and $B_X$, respectively. We use the notion $X^*$ for the dual space of $X$, and we write by $\Pi(X)$ the \emph{state} of $X$ given by $\Pi(X) = \{(x,x^*) \in S_X \times S_{X^*}: x^*(x)=1\}$. The space of all bounded linear operators from $X$ into $Y$ is denoted by $\Lin(X,Y)$, and when the range space coincides with the domain space, we abbreviate by $\Lin(X)$ for convenience.

A bounded linear operator $T \in \Lin(X,Y)$ is said to \emph{attain its norm} at $x_0 \in S_X$ if $\|T\|=\|Tx_0\|$, and we write by $T \in \NA(X,Y)$. We define the \emph{numerical radius} $\nu(T)$ of an operator $T \in \Lin(X)$ by
$$
\nu(T) := \sup\{|x^*(Tx)|: (x,x^*) \in \Pi(X) \}.
$$
In a similar sense, $T \in \Lin(X)$ is said to \emph{attain its numerical radius} at $(x_0,x_0^*) \in \Pi(X)$ if $\nu(T)=|x_0^*(Tx_0)|$, and we write by $T \in \NRA(X)$.

With respect to above notions, an analogous version in terms of minimum norm attaining operators are introduced recently in \cite{CN}. The \emph{minimum norm} $m(T)$ of an operator $T \in \Lin(X,Y)$ is defined by
$$
m(T) := \inf \{\|Tx\|: x \in S_X\},
$$
and the \emph{Crawford number} $c(T)$ of $T \in \Lin(X)$ \cite{C3} is defined by
$$
c(T) := \inf \{|x^*(Tx)|: (x,x^*) \in \Pi(X) \}.
$$
An operator $T \in \Lin(X,Y)$ is said to \emph{attain its minimum norm} at $x_0$ if $m(T)=\|Tx_0\|$, and we write by $T \in \MA(X,Y)$. In the same manner, one can define the Crawford number attainment of an operator as follows.

\begin{definition}
A bounded linear operator $T \in \Lin(X)$ is said to \emph{attain its Crawford number at $(x_0,x_0^*) \in \Pi(X)$} if $c(T) = |x_0^*(Tx_0)|$, and the set of all bounded operators on $X$ which attain their Crawford number is denoted by 
$\CNA(X)$. 
\end{definition}

As aforementioned, our aim is to study the behavior of the set $\CNA(X)$ as $X$ varies. For instance, one can easily deduce that every operator attains its Crawford number when $X$ is finite-dimensional (Proposition \ref{prop:finite-dimensional}). We will see many properties that $\CNA(X)$ have including the case of adjoint operators and the set restricted by compact operators. In Section \ref{section:CNA}, we look also for a sufficient condition of $X$ having that $\CNA(X)$ is dense in $\Lin(X)$, such as when $X$ has the RNP. In addition, we will show that the Bishop-Phelps theorem for bounded closed convex sets holds when $X$ is a real Banach space. At the last of the section, we introduce some Banach spaces satisfying a weaker version of the Bishop-Phelps-Bollob\'as property for the Crawford number, which is a stronger condition compared to the denseness of Crawford number attaining operators. Finally in Section \ref{section:question}, we list up some difficulties in dealing with considerable results concerning the Crawford number.

\section{On the Crawford number attaining operators}\label{section:CNA}

This section basically consists of three subsections: on the set of Crawford number operators, on the denseness of Crawford number operators and on the Bishop-Phelps-Bollob\'as property for the Crawford number. More precisely, we first see some set relations on $\CNA(X)$, and then investigate the denseness in two ways when the set $\CNA(X)$ is dense in the whole space or when it satisfies a stronger condition compared to the denseness.

\subsection{On the set of Crawford number attaining operators}

We will look for the case when $\CNA(X)$ coincides with $\Lin(X)$ or when it is not. The very first result covers the finite-dimensional case for $X$, which follows from an immediate fact.

\begin{prop}\label{prop:finite-dimensional}
Let $X$ be a finite-dimensional Banach space. Then, $\CNA(X) = \Lin(X)$.
\end{prop}

\begin{proof}
Since $\Pi(X)$ is a compact set in $S_X \times S_{X^*}$, the infimum can be chosen to be the minimum for any $T \in \Lin(X)$.
\end{proof}

Recall that an operator $T \in \Lin(X)$ is called a \emph{monomorphism} if $T$ is injective and $T(X)$ is closed, or equivalently, if there exists a constant $\lam>0$ such that $\|Tx\| \geq \lam\|x\|$ for all $x \in S_X$. Being aware of that an operator is whether injective or a monomorphism helps us to determine related consequence in view of the Crawford number.

\begin{prop} We have the followings.
\begin{enumerate}
\item[\textup{(a)}] If an operator $T \in \Lin(X)$ is not a monomorphism, then $c(T)=0$. Thus the Crawford number of a compact operator is zero for an infinite-dimensional Banach space $X$.
\item[\textup{(b)}] If $T \in \Lin(X)$ is not injective, then $T \in \CNA(X)$. However, the converse does not hold even if $c(T)=0$.
\end{enumerate}
\end{prop}

\begin{proof}
(a) Since $T$ is not a monomorphism, there is a sequence $\{x_n\} \subseteq S_X$ such that $\lim_n\|Tx_n\|=0$, which implies that $m(T)=0$. Thus for every $\eps>0$ there exists $x_n$ such that $\|Tx_n\|<\eps$ and choose $x_n^*\in S_{X^*}$ such that $(x_n,x_n^*) \in \Pi(X)$. Then $|x_n^*(Tx_n)|\leq \|Tx_n\|<\eps$. (b) One way is clear. For the converse, just consider $T \in \Lin(\mathbb{K}^2)$ given by $T((u,v)):=(v,u)$ for $u,v \in \mathbb{K}$. Then we have $x_0^*(Tx_0)=0$ for $(x_0,x_0^*)=((1,0),(1,0)) \in \Pi(\mathbb{K}^2)$ while $T$ is injective.
\end{proof}

We cannot conclude as well on the contrary that $T \in \Lin(X)$ attains its Crawford number when $T$ is injective and $c(T)=0$. To see this, we refer to the following result. Here we will denote by $\CNA_\K(X)$ the set of all compact operators which attain their Crawford number. Recall that an unconditional basis $\{x_n\}$ of $X$ is called \emph{1-unconditional} \cite{LT} if $\sup_{\theta  \in \{-1, 1\}^\mathbb{N}}\|S_\theta\|=1$, where $S_\theta \in \Lin(X)$ is defined by $S_\theta\bigl(\sum_{n=1}^\infty a_nx_n \bigr) := \sum_{n=1}^\infty a_n\theta_nx_n$ for each $\theta=\{\theta_n\}_{n=1}^\infty \in \{-1, 1\}^\mathbb{N}$

\begin{prop}\label{prop:compact-CNA}
Let $X$ be a Banach space with a 1-unconditional basis. Then, $\CNA_\K(X) \neq \K(X)$. In particular, there is an injective operator $T \in \Lin(X)$ with $c(T)=0$ which does not attain its Crawford number.
\end{prop}

\begin{proof}
Let $\{x_n\} \subseteq X$ be a 1-unconditional basis of $X$. Let $\{x_n^*\} \subseteq X^*$ be a coordinate functionals such that $x_m^*(x_n)=\delta_{n,m}$ where $\delta_{n,m}$ is the Kronecker delta. We shall write $x \in X$ and $x^* \in X^*$ by the form $x = \sum_{n=1}^\infty a_nx_n \in X$ and $x^* = \text{weak}^*\text{-} \lim_N\sum_{n=1}^N a_n^*x_n^* \in X^*$. Define an operator $T: X \to X$ linearly by $Tx_n = \dfrac{x_n}{n}$ for each $n \in \N$. It is clear that $T$ is injective. If we let $T_k:X\to X$ be the finite rank operator defined for each $k\in \mathbb{N}$ by  
$$
T_k(x) = \sum_{n=1}^k x_n^*(Tx) x_n \qquad \text{for } x\in X,
$$
then we have
\begin{align*}
\|T-T_k\| &= \sup_{\sum a_nx_n \in B_X} \left\| \sum_{n=k+1}^\infty \frac{a_nx_n}{n} \right\| \\
&\leq \frac{2}{k+1} \sup_{\sum a_nx_n \in B_X} \left\|\sum_{n=k+1}^\infty a_nx_n\right\| \\
&\leq \frac{2}{k+1}
\end{align*}
by \cite[Proposition 1.c.7]{LT}. So $T_k$ converges to $T$, and this shows that $T \in \K(X)$. However, there is no $(x_0,x_0^*) \in \Pi(X)$ such that $x_0^*(Tx_0)=c(T)=0$. Indeed, if we let $x_0=\sum_{n=1}^\infty a_nx_n$ and $x_0^*=\text{weak}^*\text{-} \lim_N \sum_{n=1}^N a_n^*x_n^*$, then we have $\sum_{n=1}^\infty a_n^*a_n=1$. It follows by the 1-unconditionality of $\{x_n\}$ that $\re a_n^*a_n\geq 0$ for all $n\in \N$. To see this, assume there is $n_0 \in \N$ such that $\re a_{n_0}^*a_{n_0}< 0$. Since $\{x_n\}$ is 1-unconditional, we have
$$
z_0:=\left(\sum_{n\neq n_0} a_nx_n\right) - a_{n_0}x_{n_0} \in B_X.
$$
However, we obtain from assumption that $\re x_0^*(z_0)>1$, which is a contradiction to the choice of $(z_0,x_0^*) \in B_X \times S_{X^*}$. Thus this implies that
$$
0 < \re \sum_{n=1}^\infty \frac{a_n^*a_n}{n} = \re x_0^*(Tx_0) \leq |x_0^*(Tx_0)|,
$$
finishing the proof.
\end{proof}

We now would like to see a characterization of Crawford number attaining adjoint operators. This is a continuation of works done in \cite{AP}, in the sight of the Crawford number. We begin with a fundamental result of the Crawford number for adjoint operators.

\begin{lemma}\label{lemma:adjoint}
Let $X$ be a Banach space. For $T \in \Lin(X)$, we have $c(T)=c(T^*)$.
\end{lemma}

\begin{proof}
$c(T^*) \leq c(T)$ is clear. For the converse, let $\gamma>0$ be given and assume that $\|T\|=1$. By definition of $c(T^*)$, there exists $(x_0^*,x_0^{**}) \in \Pi(X^*)$ such that $|x_0^{**}(T^*x_0^*)|<c(T^*) + \dfrac{\gamma}{4}$. We may find a net $\{x_\alpha\}_{\alpha \in I} \subseteq S_X$ which converges to $x_0^{**}$ in the weak$^*$ topology by the Goldstine theorem. Let us fix $\alpha_0 \in I$ such that
$$
|J(x_{\alpha_0})(T^*x_0^*) - x_0^{**}(T^*x_0^*)| < \frac{\gamma}{4} \quad \text{and} \quad \re \left[ J(x_{\alpha_0})(x_0^*) \right]>1-\frac{1}{4} \left(\frac{\gamma}{4}\right)^2,
$$
where $J: X \to X^{**}$ is the canonical embedding. Since $\re x_0^*(x_{\alpha_0}) > 1-\dfrac{1}{4} \left(\dfrac{\gamma}{4}\right)^2$, by the Bishop-Phelps-Bollob\'as theorem \cite{B1}, there exists $(y_0,y_0^*) \in \Pi(X)$ such that
$$
\|y_0-x_{\alpha_0}\| < \frac{\gamma}{4} \quad \text{and} \quad \|y_0^*-x_0^*\|<\frac{\gamma}{4}.
$$
Thus we can deduce that
\begin{align*}
|y_0^*(Ty_0) - x_0^*(Tx_{\alpha_0})| &\leq |y_0^*(Ty_0) - y_0^*(Tx_{\alpha_0})| + |y_0^*(Tx_{\alpha_0}) - x_0^*(Tx_{\alpha_0})| \\
&\leq \|y_0-x_{\alpha_0}\| + \|y_0^*-x_0^*\| < \frac{\gamma}{2}.
\end{align*}
Finally, it follows that
\begin{align*}
|y_0^*(Ty_0)| &\leq |y_0^*(Ty_0) - x_0^*(Tx_{\alpha_0})| + |J(x_{\alpha_0})(T^*x_0^*) - x_0^{**}(T^*x_0^*)| + |x_0^{**}(T^*x_0^*)|\\
&< \frac{\gamma}{2} + \frac{\gamma}{4} + \left(c(T^*) + \frac{\gamma}{4}\right) = c(T^*) + \gamma.
\end{align*}
Since $\gamma>0$ was arbitrary, we have shown that $c(T) \leq c(T^*)$.
\end{proof}

In \cite{AP}, they showed a characterization result of numerical radius attaining adjoint operators in order to show that the set of operators whose adjoint attain their numerical radius is dense in $\Lin(X)$ for every Banach space $X$. We would like to follow their idea of the proof in case of the Crawford number. Our main curiosity, concerning the denseness of operators whose adjoint attain their Crawford number, is stated in Question \ref{question:adjoint-dense}.

\begin{prop}\label{prop:CNA-equivalent}
Let $X$ be a Banach space. For $T \in \Lin(X)$, the following are equivalent.
\begin{enumerate}
\item[\textup{(a)}] $T^* \in \CNA(X^*)$.
\item[\textup{(b)}] There are sequences $\{x_n\} \subseteq S_X$, $\{x_n^*\} \subseteq S_{X^*}$ and $\{\delta_n\}, \{\eps_n\} \subseteq \R^+$ satisfying
\begin{enumerate}
\item[\textup{(i)}] $\lim_n \delta_n = \lim_n (\eps_n/\delta_n) = 0$,
\item[\textup{(ii)}] $1 + \delta_n |x_{n+k}^*(Tx_n)| \leq |x_{n+k}^*(x_n)| + \delta_n c(T) + \eps_n$.
\end{enumerate}
\end{enumerate}
\end{prop}

\begin{proof}
(a)$\Rightarrow$(b). For $T^* \in \CNA(X^*)$, let $(x_0^*,x_0^{**}) \in \Pi(X^*)$ be such that $|x_0^{**}(T^*x_0^*)|=c(T^*)=c(T)$ by Lemma \ref{lemma:adjoint}. Let $\{\delta_n\}, \{\eps_n\}$ be sequences in $\R^+$ satisfying (i). Take $x_n^* := x_0^*$, then by the Goldstine theorem, we may find a sequence $\{x_n\} \subseteq S_X$ such that
$$
|x_0^*(Tx_n) - x_0^{**}(T^*x_0^*)| \leq \frac{\eps_n}{2\delta_n} \qquad \text{and} \qquad |x_0^*(x_n) - x_0^{**}(x_0^*)| \leq \frac{\eps_n}{2}.
$$
Then we have
$$
|x_0^*(x_n)| \geq 1 - \frac{\eps_n}{2} \qquad \text{and} \qquad |x_0^*(Tx_n)| - \frac{\eps_n}{2\delta_n} \leq c(T).
$$
Hence
$$
1 + \delta_n |x_0^*(Tx_n)| \leq |x_0^*(x_n)| + \delta_n c(T) + \eps_n,
$$
obtaining (ii).

(b)$\Rightarrow$(a). As $\{x_n\} \subseteq S_X$ and $\{x_n^*\} \subseteq S_{X^*}$, there are cluster points $x_0^{**} \in B_{X^{**}}$ and $x_0^* \in B_{X^*}$ of their weak$^*$ topologies, respectively. By (ii), we have
$$
1 \leq |x_{n+k}^*(x_n)| + \delta_n c(T) + \eps_n.
$$
Letting $k$ tend to infinity, we can deduce from (i) that
$$
1 \leq |x_0^{**}(x_0^*)| + \delta_n c(T) + \eps_n.
$$
It follows by letting $n$ tend to infinity that $|x_0^{**}(x_0^*)| \geq 1$. Now, again by (ii), we obtain
$$
|x_{n+k}^*(Tx_n)| \leq c(T) + \frac{\eps_n}{\delta_n}
$$
for all $n$ and $k$. Thus we have $|x_0^{**}(T^*x_0^*)| \leq c(T)$, which shows that $T^*$ attains its Crawford number.
\end{proof}

\subsection{Denseness of Crawford number attaining operators}

We are now interested in how many operators attain their Crawford number in $\Lin(X)$, which leads naturally to the denseness problem of the set $\CNA(X)$. First, we introduce a very basic, but crucial fundamental result which will be frequently used throughout the article.

\begin{prop}\label{prop:radius-0-approx}
Let $X$ be a Banach space. Then, every operator $T \in \Lin(X)$ with $c(T)=0$ can be approximated by operators which attain their Crawford number.
\end{prop}

\begin{proof}
Let $\eps>0$ and an operator $T \in \Lin(X)$ with $c(T)=0$ be given. Then, there exists $(x_0,x_0^*) \in \Pi(X)$ such that $|x_0^*(Tx_0)|<\eps$. Define an operator $S \in \Lin(X)$ by $Sx := Tx - x_0^*(Tx_0) x$. It is clear that
$$
\|S-T\| = \sup_{x \in B_X} \|x_0^*(Tx_0) x\| <\eps
$$
and
$$
x_0^*(Sx_0)=x_0^*(Tx_0)-x_0^*(Tx_0) x_0^*(x_0)=0,
$$
which shows that $S \in \CNA(X)$.
\end{proof}

Recall from Proposition \ref{prop:compact-CNA} that $\CNA_\K(X) \neq \K(X)$ for every infinite-dimensional $X$ with a 1-unconditional basis. On the contrary, the next result shows that in the case of compact operators, $\CNA_\K(X)$ is always dense in $\K(X)$.

\begin{prop}
Let $X$ be a Banach space. Then, $\CNA_\K(X)$ is dense in $\K(X)$.
\end{prop}

\begin{proof}
Let $\eps>0$ and an operator $T \in \K(X)$ be given. Assume $X$ is infinite-dimensional, and since we know that $m(T)=0$ due to the argument given in \cite[Corollary 2.4]{C2}, there exists $x_0 \in S_X$ such that $\|Tx_0\|<\eps$. Define an operator $S \in \Lin(X)$ by $Sx := Tx - x_0^*(x) Tx_0$ where $x_0^*$ is chosen so that $(x_0,x_0^*) \in \Pi(X)$. Then, it is clear that $\|S-T\|<\eps$ and $x_0^*(Sx_0)=0$, which shows that $S \in \CNA(X)$. It is easy to see that $S$ is compact since it is a rank-one perturbation of $T \in \K(X)$.
\end{proof}

If we go back to the minimum norm attaining operators, it is clear that minimum version of the Bishop-Phelps theorem holds for every Banach space $X$. This is a direct consequence of the failure of injectivity when $\dim (X) \geq 2$ according to \cite{C2}, and the case $\dim(X)=1$ is clear. Restricting $X$ to be a real Banach space, we are also able to obtain a minimum analogue of the Bishop-Phelps theorem for bounded closed convex set \cite{BP}. This may help giving some idea to Question \ref{question:beta} concerning the Crawford number.

\begin{prop}\label{prop:min-attaining}
Let $C$ be a nonempty bounded closed convex subset of a real Banach space $X$. Then, the set of functionals $x^* \in X^*$ such that $|x^*|$ attain their minimum on $C$ is dense in $X^*$.
\end{prop}

\begin{proof} Let $x_0^* \in X^*$ and $\eps>0$ be given. First, if $\inf_{x \in C} |x_0^*(x)|=0$, let us define $A=\{ x\in C : |x_0^*(x)|<\epsilon\}$. If $0$ is a limit point of $A$, then $C$ contains $0$ and  $|x_0^*|$ attains its minimum $0$ at $0$. So we may assume that $\lam :=\inf \{ \|x\| : x\in A\}>0$. Fix $x_0\in A$ such that $|x_0^*(x_0)|<\lam\eps$ and choose $z^* \in S_{X^*}$ such that $ z^*(x_0) =\|x_0\|$. If we consider $x^* \in X^*$ given by
$$
x^*(x) := x_0^*(x) - \frac{z^*(x)}{z^*(x_0)} x_0^*(x_0) \qquad \text{for } x \in X,
$$
then $\|x^*-x_0^*\|\leq \eps$ and $|x^*|$ attains its minimum 0 at $x_0 \in C$.

So suppose that $\inf_{x \in C} |x_0^*(x)| >0$. By the connectedness of $C$, we can see that either $x_0^*(x) >0$ for all $x \in C$ or $x_0^*(x) <0$ for all $x \in C$. Assume first that $x_0^*(x)>0$ for all $x \in C$. Then, it follows that $-x_0^*(x)<0$ for all $x \in C$, and thus by the Bishop-Phelps theorem for bounded closed convex set there exists $z^* \in X^*$ with $\|z^*\|<\min \{ \eps, \inf_{x \in C} x_0^*(x) \}$ such that $-x_0^*+z^*$ attains its maximum on $C$. Arguing from the fact that $(-x_0^*+z^*)(x) < 0$ for all $x \in C$, we can deduce that $|x_0^*-z^*|$ attains its minimum on $C$. On the other hand, if $x_0^*(x) <0$ for all $x \in C$, then we can find $z^* \in X^*$ with $\|z^*\|<\min\{ \eps, -\inf_{x \in C} x_0^*(x)\}$ such that $x_0^*+z^*$ attains its maximum on $C$. Similarly, we have that $|x_0^*+z^*|$ attains its minimum on $C$.
\end{proof}

Now, let us consider a specific case of $X$ when it has the RNP. Recall that the Radon-Nikod\'ym property (RNP in short) is a very important concept of a Banach space which provides many geometric reformulations of a Banach space; for instance, it was shown in \cite{B2} that the RNP is closely related to the denseness of the set of norm attaining operators. We refer to \cite[p.217]{DU} for a plenty of characterizations associated to the RNP. In \cite{C2}, the author showed that if either $X$ or $Y$ has the RNP, then $\MA(X,Y)$ is dense in $\Lin(X,Y)$ for every Banach space $Y$. On the one hand, it is shown in \cite{AP} that if $X$ has the RNP, then $\NRA(X)$ is dense in $\Lin(X)$. Before we proceed to demonstrate our main result, we shall remind of a well-known optimization lemma on a space with the RNP.

Recall that a nonempty bounded closed convex set $D \subseteq X$ is an \emph{RNP set} if every subset of $D$ is dentable (see \cite{B2}). A Banach space has the RNP if its unit ball has the RNP, so every bounded closed convex subset of a Banach space with the RNP is an RNP set. Recall also that a functional $x^* \in X^*$ \emph{strongly exposes} $D \subseteq X$ if there exists a point $x_0 \in D$ such that $x^*(x_0) = \sup_{x \in D} x^*(x)$ and $\{x_n\}_{n=1}^\infty \subseteq D$ converges to $x_0$ whenever $\lim_n x^*(x_n)=x^*(x_0)$. The following celebrated theorem says that there are densely many strongly exposing functionals for an RNP set.

\begin{lemma}[\mbox{Stegall's optimization principle, \cite[Theorem 14]{S}}]\label{lemma:optimization}
Let $D$ be a bounded RNP set of a Banach space $X$ and $\phi: D \to [-\infty, \infty)$ be an upper semicontinuous and bounded above function. If $\phi$ is not identically $-\infty$ on $D$, then the set
$$
\{ x^* \in X^*: \phi + \re x^* \textup{ strongly exposes } D \}
$$
is a $G_\delta$-dense subset of $X^*$.
\end{lemma}

For the next lemma, we follow the proof of \cite[Lemma 2.3]{AP}.

\begin{lemma}\label{lemma:semicontinuous}
Let $X$ be a Banach space and $T \in \Lin(X)$ with $c(T)>0$. Define $\phi_T: B_X \to (-\infty, \infty]$ by
\begin{displaymath}
\phi_T(x)=\left\{\begin{array}{@{}cl}
\displaystyle \phantom{.} \frac{1}{\|x\|}\min \left\{\left|x^*\left(T\left(\frac{x}{\|x\|}\right)\right)\right|: \left(\frac{x}{\|x\|},x^*\right) \in \Pi(X)\right\} & \text{if } x \neq 0 \\\\
\infty & \text{if } x=0 \\
\end{array} \right.
\end{displaymath}
Then, $\phi_T$ is bounded below and lower semicontinuous.
\end{lemma}

\begin{proof} We may assume that $\|T\|=1$. It is evident that $\phi_T$ is bounded below. In order to show that it is lower semicontinuous, we claim that the set $$D=\{x\in B_X : \phi_T(x)\le s\}$$ is closed for each $0 < s \leq \infty$. The case $s = \infty$ is clear, so fix any $0<s<\infty$. Assuming $D$ has a limit point, let $x_0$ be a limit point of $D$ and choose a sequence $\{x_n\} \subseteq B_X\setminus \{0\}$ converging to $x_0$ such that $\phi_T(x_n) \leq s$ for all $n \in \N$.

Since $\phi_T(x) \|x\|\ge c(T)>0$ for each $x \in B_X\setminus \{0\}$, we have $s\|x_n\|\ge c(T)>0$ for all $n$. Hence $s\|x_0\|\ge c(T)>0$. So  $x_0\neq 0$.  

We write $y_n = \dfrac{x_n}{ \|x_n\|}$, $y_0 =\dfrac{ x_0}{\|x_0\|}$ and choose each $x_n^* \in S_{X^*}$ so that $x_n^*(y_n)=1$ and $|x_n^*(Ty_n)|=\|x_n\|\phi_T(x_n)$. If we let $x_0^*$ be a weak$^*$-cluster point of $\{x_n^*\}$, then
$$
|1-x_0^*(y_0)| = |x_n^*(y_n)-x_0^*(y_0)| \leq \|y_n-y_0\| + |x_n^*(y_0) - x_0^*(y_0)|
$$
for all $n$, so $x_0^*(y_0)=1$. Also, we have
$$
|x_n^*(Ty_n)-x_0^*(Ty_0)| \leq \|y_n-y_0\| + |x_n^*(Ty_0)-x_0^*(Ty_0)|,
$$
and $\{x_n^*(Ty_n)\}$ converges to $x_0^*(Ty_0)$. Since
$$
\frac{|x_n^*(Ty_n)|}{\|x_n\|} = \phi_T(x_n) \leq s,
$$
we can derive that $\phi_T(x_0) \leq \frac{ |x_0^*(Ty_0)|}{\|x_0\|} \leq s$, which proves the claim.
\end{proof}

Consequently, the following result gives a positive answer to the denseness of Crawford number attaining operators in many cases of Banach spaces.

\begin{theorem}\label{theorem:RNP}
Let $X$ be a Banach space with the RNP. Then, $\CNA(X)$ is dense in $\Lin(X)$.
\end{theorem}

\begin{proof}
Let $\eps>0$ and $T \in \Lin(X)$ be given. If $c(T)=0$, then $T$ can be approximated by $S \in \CNA(X)$ due to Proposition \ref{prop:radius-0-approx}. So we may assume that $c(T) > 0$. By Lemma \ref{lemma:optimization}, there exists $x_0 \in B_X$ and $z^* \in X^*$ with $0<\|z^*\|<\min\{\eps,c(T)\}$ such that
$$
- \phi_T(x_0) + \re z^*(x_0)\geq - \phi_T(x) + \re z^*(x) \qquad \text{for all } x\in B_X,
$$  
where $\phi_T$ is a function defined in Lemma \ref{lemma:semicontinuous}. By a suitable rotation, the above equation can be changed into
$$
-\phi_T(x_0) + \re z^*(x_0) \geq - \phi_T(x) + |z^*(x)| \qquad \text{for all } x \in B_X.
$$
Hence $x_0\in S_X$ and $ \re z^*(x_0) =|z^*(x_0)|$. Choose $x_0^* \in S_{X^*}$  satisfying $x_0^*(x_0)=1$ and $\phi_T(x_0) = |x_0^*(Tx_0)|$.

Now, define an operator $S \in \Lin(X)$ by $S(x) := Tx - \lam z^*(x) x_0$ where $\lam \in S_\mathbb{K}$ is chosen so that $ \lam|x_0^*(Tx_0)| = x_0^*(Tx_0) $. It follows that
$$
\|S-T\| = \sup_{x \in B_X} |\lam z^*(x)| \|x_0\| < \eps.
$$
Moreover, $S$ attains its Crawford number. Indeed, for every $(x,x^*) \in \Pi(X)$, we have
\begin{align*}
|x^*(Sx)| \geq |x^*(Tx)| - |\lam z^*(x) x^*(x_0)| &\geq \phi_T(x) - |z^*(x)| \\
&\geq \phi_T(x_0) - |z^*(x_0)|.
\end{align*}
But on the other hand, we have from $\|z^*\|<c(T)$ that
$$
|x_0^*(Sx_0)| = \bigl| \lam |x_0^*(Tx_0)| - \lam z^*(x_0) x_0^*(x_0) \bigr| = |x_0^*(Tx_0)| -|z^*(x_0)|,
$$
so $S$ attains its Crawford number at $(x_0,x_0^*) \in \Pi(X)$.
\end{proof}



\subsection{The Bishop-Phelps-Bollob\'as property for the Crawford number}

In \cite{KLM2}, the authors defined a notion of the Bishop-Phelps-Bollob\'as property for numerical radius, which is a stronger condition allowing one to approximate operator and the state which ``almost" attains its numerical radius in both ways. In a similar spirit, we can define a somewhat weaker notion in terms of the Crawford number which was also defined in \cite{KLM2}. More precisely, we approximate both operator and the state simultaneously in the same sense, but it may not preserve the Crawford number between the original operator and the new operator.
\begin{definition}
A Banach space $X$ is said to have the \emph{weak Bishop-Phelps-Bollob\'as property for the Crawford number} if for every $\eps>0$, there exists $\eta(\eps)>0$ such that whenever $T \in \Lin(X)$ and $(x,x^*) \in \Pi(X)$ satisy that $|x^*(Tx)|<c(T)+\eta(\eps)$, we can find $S \in \CNA(X)$ and $(z,z^*) \in \Pi(X)$ such that
$$
|z^*(Sz)|=c(S), \quad \|S-T\|<\eps, \quad \|z-x\|<\eps \quad \text{and} \quad \|z^*-x^*\|<\eps.
$$
\end{definition}

Observe that if $c(T)=0$, then the argument of Proposition \ref{prop:radius-0-approx} gives that we are able to find a new operator $S$ without even perturbating $(x,x^*) \in \Pi(X)$. We now prove the minimum analogue of \cite[Proposition 4]{KLM1}. Although the argument is similar, we give a detail of its proof for completeness. For the discussion on natural extension of the Bishop-Phelps-Bollob\'as property for the Crawford number, we refer to Section \ref{subsection:BPBp}.

Recall that a Banach space $X$ is said to be \emph{uniformly convex} if the \emph{modulus of convexity} $\delta_X$ of $X$ defined by
$$
\delta_X(\eps) := \inf \left\{ 1 - \frac{\|x+y\|}{2} : x,y \in B_X, \|x-y\| \geq \eps \right\}
$$
is positive for every $0<\eps \leq2$. A Banach space $X$ is \emph{uniformly smooth} if the \emph{modulus of smoothness} $\rho_X$ of $X$ defined by
$$
\rho_X(\tau) := \sup \left\{ \frac{\|x+\tau y\| + \|x-\tau y\|}{2} - 1 : x, y \in S_X \right\}
$$
satisfies that $\lim_{\tau \to 0} \rho(\tau)/\tau=0$. It is well-known that every uniformly convex is reflexive (hence it has the RNP) and that $X$ is uniformly convex (resp. uniformly smooth) if and only if $X^*$ is uniformly smooth (resp. uniformly convex).

\begin{prop}\label{prop:weak-BPBp}
Let $X$ be a uniformly convex and uniformly smooth Banach space. Then, $X$ has the weak Bishop-Phelps-Bollob\'as property for the Crawford number.
\end{prop}

\begin{proof}
Let $\eps>0$ be given. If we write by $\delta_X(\eps)$ the modulus of convexity of $X$, we may set a constant $\eta$ depending on $\eps$ by
$$
\eta(\eps) := \frac{\eps}{4} \min \left\{\delta_X\left(\frac{\eps}{4}\right),\delta_{X^*}\left(\frac{\eps}{4}\right)\right\}>0
$$
since $X^*$ is also uniformly convex. Suppose that $T \in \Lin(X)$ and $(x,x^*) \in \Pi(X)$ satisfy $|x^*(Tx)|<c(T)+\eta(\eps)$. We put $(T_0,x_0,x_0^*):=(T,x,x^*)$ and assume that $T_n$ is constructed. Define $T_{n+1} \in \Lin(X)$ inductively by
$$
T_{n+1}(x) := T_nx + \lam_{n+1} \frac{\eps^{n+1}}{4^{n+1}} x_n^*(x)x_n
$$
for each $x \in X$, where $\lam_n$ is chosen so that
$$
|\lam_n|=1 \quad \text{and} \quad x_n^*(T_nx_n) = -\lam_n|x_n^*(T_nx_n)|.
$$
Here, we choose $(x_n,x_n^*) \in S_X \times S_{X^*}$ with $|x_n^*(x_n)|=1$ such that $x_n^*(x_{n-1})=|x_n^*(x_{n-1})|$ and
$$
|x_n^*(T_nx_n)| \leq c(T_n) + \eta\left(\frac{\eps^{n+1}}{4^{n+1}}\right).
$$
Then we obtain for each $n \in \N$ that
$$
\|T_{n+1} - T_n\| \leq \frac{\eps^{n+1}}{4^{n+1}}.
$$
It follows from above that $\{T_n\}$ is convergent, say to $S \in \Lin(X)$. Notice that $c(T_n)$ are convergent since $c(\cdot)$ is continuous with respect to the operator norm. Indeed, given operators $P, Q\in \mathcal{L}(X)$ and for each $(x, x^*)\in \Pi(X)$, 
\[ c(P) \leq |x^*(Px)| \leq |x^*(Px-Qx)| + |x^*(Qx)| \leq \|P-Q\| + |x^*(Qx)|.\]
By taking infimum, we get $c(P) \leq \|P-Q\| + c(Q)$. Hence
\[|c(P)-c(Q)| \leq ||P-Q\|,\]
so $\{c(T_n)\}$ also converges. It remains to show that sequences $\{x_n\}$ and $\{x_n^*\}$ are both convergent to some points $x_\infty$ and $x_\infty^*$ respectively so that $S$, $x_\infty$ and $x_\infty^*$ satisfy the desired condition. Indeed, we have from the construction that
\begin{align*}
c(T_{n+1}) +\eta\left(\frac{\eps^{n+2}}{4^{n+2}}\right) &\geq |x_{n+1}^*(T_{n+1}x_{n+1})| \\
&= \left|x_{n+1}^*(T_nx_{n+1}) + \lam_{n+1} \frac{\eps^{n+1}}{4^{n+1}} x_n^*(x_{n+1})x_{n+1}^*(x_n)\right| \\
&\geq |x_{n+1}^*(T_nx_{n+1})| - \frac{\eps^{n+1}}{4^{n+1}} |x_{n+1}^*(x_n)| \\
&\geq c(T_n) - \frac{\eps^{n+1}}{4^{n+1}} x_{n+1}^*(x_n),
\end{align*}
and also that
\begin{align*}
c(T_{n+1}) &\leq |x_n^*(T_{n+1}x_n)| \\
&= \left|x_n^*(T_nx_n) + \lam_{n+1} \frac{\eps^{n+1}}{4^{n+1}}\right| \\
&= |x_n^*(T_nx_n)| - \frac{\eps^{n+1}}{4^{n+1}} \\
&\leq c(T_n) + \eta\left( \frac{\eps^{n+1}}{4^{n+1}} \right) - \frac{\eps^{n+1}}{4^{n+1}}.
\end{align*}
These sum up to an equation
\begin{align*}
x_{n+1}^*(x_n) &\geq 1 - \frac{4^{n+1}}{\eps^{n+1}} \left[ \eta\left( \frac{\eps^{n+1}}{4^{n+1}} \right) + \eta\left(\frac{\eps^{n+2}}{4^{n+2}}\right) \right] \\
&\geq 1- 2\,\frac{4^{n+1}}{\eps^{n+1}} \eta\left( \frac{\eps^{n+1}}{4^{n+1}} \right)
\end{align*}
since $\eta$ is decreasing. Thus by uniform convexity, we have
$$
\left\| \frac{x_n+x_{n+1}}{2} \right\| \geq x_{n+1}^* \left( \frac{x_n+x_{n+1}}{2} \right) \geq 1- \delta_X \left(\frac{\eps^{n+1}}{4^{n+2}}\right)
$$
and also that
$$
\left\| \frac{x_n^*+x_{n+1}^*}{2} \right\| \geq \left( \frac{x_n^*+x_{n+1}^*}{2} \right) (x_n) \geq 1- \delta_{X^*} \left(\frac{\eps^{n+1}}{4^{n+2}}\right)
$$
assuming $\eps<4$. This leads to that
$$
\|x_n-x_{n+1}\| \leq \frac{\eps^{n+1}}{4^{n+2}} \quad \text{and} \quad \|x_n^*-x_{n+1}^*\| \leq \frac{\eps^{n+1}}{4^{n+2}}.
$$
Putting $x_n \to x_\infty$ and $x_n^* \to x_\infty^*$, we will see our condition is satisfied.

First,
$$
\|S-T\| \leq \sum_{n=0}^\infty \frac{\eps^{n+1}}{4^{n+1}} \leq \frac{\eps}{2} < \eps
$$
is clear. Next, if we set $(z,z^*) \in \Pi(X)$ by $z=x_\infty$ and $z^*= \frac{x_\infty^*}{x_\infty^*(x_\infty)}$, then we have
$$
c(S) = \lim_n c(T_n) = \lim_n |x_n^*(T_nx_n)| = |z^*(Sz)|,
$$
$$
\|z-x\| \leq \sum_{n=0}^\infty \frac{\eps^{n+1}}{4^{n+2}} \leq \frac{\eps}{2} < \eps
$$
and
\begin{align*}
\|z^*-x^*\| &\leq \left|1-\frac{1}{x_\infty^*(x_\infty)}\right| + \sum_{n=0}^\infty \frac{\eps^{n+1}}{4^{n+2}} \\
&\leq |x_\infty^*(x_\infty)-x^*(x_\infty)| + |x^*(x_\infty) - x^*(x)| + \frac{\eps}{2} \\
&< \frac{\eps}{4} + \frac{\eps}{4} + \frac{\eps}{2} = \eps,
\end{align*}
which finishes the proof.
\end{proof}

\section{Questions \& Remaining Problems}\label{section:question}

In this section, we discuss some difficulties and considerable problems dealing with the denseness of Crawford number attaining operators.

\subsection{Denseness of adjoint operators}

According to Proposition \ref{prop:CNA-equivalent}, it is possible to characterize the Crawford number attainment for adjoint operators in terms of existence of sequences with certain conditions. However, the minimum anaologue might not work because the change of positions on assertions in Proposition \ref{prop:CNA-equivalent} induces a sign-related problem. So our main question remaining is the following.

\begin{question}\label{question:adjoint-dense}
Let $X$ be a Banach space. Is the set
$$
\{T \in \Lin(X): T^* \in \CNA(X^*)\}
$$
dense in $\Lin(X)$?
\end{question}

\subsection{Property $\alpha$}

Our next approach is to determine the denseness of Crawford number attaining operatos when $X$ has so-called property $\alpha$. A Banach space $X$ is said to have \emph{property $\alpha$} if there exist $\{x_\alpha\}_{\alpha \in I} \subseteq S_X$ and $\{x_\alpha^*\}_{\alpha \in I} \subseteq S_{X^*}$ with a constant $0 \leq \rho <1$ satisfying that
\begin{enumerate}
\item[\textup{(i)}] $x_\alpha^*(x_\alpha)=1 \quad \text{for all } \alpha \in I$,
\item[\textup{(ii)}] $|x_\alpha^*(x_\beta)| \leq \rho <1 \quad \text{with } \alpha,\beta \in I \text{ with } \alpha \neq \beta$,
\item[\textup{(iii)}] $B_X = \overline{\co}\{x_\alpha\}_{\alpha \in I}$, where $\overline{\co}$ denotes the closed convex hull of a set.
\end{enumerate}

One can easily show that $B_X = \overline{\co}(S)$ for some uniformly strongly exposed set $S$ (see \cite{L} for the definition of a uniformly strongly exposed set) if $X$ has property $\alpha$. Thus by following the repertoire of \cite{A1}, we may find a way to show the denseness of Crawford number attaining operators when $X$ has property $\alpha$. The main difference here, compared to ones in \cite{A1}, would be the following one.

\begin{question}
Let $X$ be a Banach space and $S \subseteq B_X$ be a subset such that $B_X=\overline{\co}(S)$. Given $\eps>0$ and $T \in \Lin(X)$, does
$$
c(T) = \inf \{ |x^*(Tx)|: (x,x^*) \in \Pi(X),\, \dist(x,S)<\eps \}
$$
hold?
\end{question}

If we are able to solve the above problem, we can prove the following desired question, arguing as same as in the proof of main theorem in \cite{A1}.

\begin{question}
Let $X$ be a Banach space such that $B_X = \overline{\co}(S)$ where $S$ is a uniformly strongly exposed set. Is $\CNA(X)$ dense in $\Lin(X)$?
\end{question}

\subsection{Property $\beta$}

Next, a consideration of the dual notion of property $\alpha$ introduced in \cite{L}, which is called a property $\beta$, can be also raised. A Banach space $X$ is said to have \emph{property $\beta$} if there exist $\{x_\alpha\}_{\alpha \in I} \subseteq S_X$ and $\{x_\alpha^*\}_{\alpha \in I} \subseteq S_{X^*}$ with a constant $0 \leq \rho <1$ satisfying that
\begin{enumerate}
\item[\textup{(i)}] $x_\alpha^*(x_\alpha)=1 \quad \text{for all } \alpha \in I$,
\item[\textup{(ii)}] $|x_\alpha^*(x_\beta)| \leq \rho <1 \quad \text{with } \alpha,\beta \in I \text{ with } \alpha \neq \beta$,
\item[\textup{(iii)}] $\|x\| = \sup_{\alpha \in I} |x_\alpha^*(x)| \quad \text{for all } x \in X$.
\end{enumerate}

Here, since the minimum analogue lacks the convexity, we do not know whether the following question can be solved in a similar way.

\begin{question}
Let $X$ be a Banach space with property $\beta$ and $\{x_\alpha\}, \{x_\alpha^*\}$ be corresponding sets. For $T \in \Lin(X)$, do we have
$$
c(T) = \inf \{|x_\alpha^*(Tx)| : \alpha \in I, (x,x_\alpha^*) \in \Pi(X) \}?
$$
\end{question}

Recall that we were able to show the minimum analogue of the Bishop-Phelps theorem for bounded closed convex sets in Proposition \ref{prop:min-attaining}. With the aid of this result, one may prove the following question if we can show the preceding inquiry.

\begin{question}\label{question:beta}
Let $X$ be a real Banach space with property $\beta$. Is $\CNA(X)$ dense in $\Lin(X)$?
\end{question}

\subsection{The Bishop-Phelps-Bollob\'as property}\label{subsection:BPBp}

In Proposition \ref{prop:weak-BPBp}, we have shown that for a special case of Banach space the weak version of the Bishop-Phelps-Bollob\'as property for the Crawford number holds. However, there is a difficulty in extending this result to the analogue of the Bishop-Phelps-Bollob\'as property for the Crawford number, as we cannot think of a kind of ``numerical index" in terms of the Crawford number (see \cite[Proposition 6]{KLM2}).

\begin{question}
Let $X$ be a uniformly convex and uniformly smooth Banach space, possibly with an additional assumption. Then for every $\eps>0$, does there exist $\eta(\eps)>0$ such that whenever $T \in \Lin(X)$ with $c(T)=1$ and $(x,x^*) \in \Pi(X)$ satisy that $|x^*(Tx)|<1+\eta(\eps)$, we can find $S \in \CNA(X)$ and $(z,z^*) \in \Pi(X)$ such that
$$
|z^*(Sz)|=c(S)=1, \quad \|S-T\|<\eps, \quad \|z-x\|<\eps \quad \text{and} \quad \|z^*-x^*\|<\eps?
$$
\end{question}

\subsection{Stability on the denseness}

Started from \cite{P}, there have been many efforts for instance in \cite{AAP,CMM,KLM1} to construct an operator which cannot be approximated by numerical radius attaining operators. However, all of those counterexamples are based on a non-injective operator, so we cannot apply the same kind of argument to the case of the Crawford number attaining operators. On the other hand, it was shown in \cite{C2} that there exists $T \in \Lin(c_0,Z)$ which cannot be approximated by minimum norm attaining operators where $Z$ is a stricty convex renorming of $c_0$. This example heavily depends on the property of isomorphism between $c_0$ and $Z$, so neither it can be applied to our case. Such difficulty gives that the following very natural question remains open.

\begin{question}
Is the set $\CNA(X)$ dense in $\Lin(X)$ for any Banach space $X$?
\end{question}

\end{document}